\documentclass[11pt,reqno]{amsart}
\usepackage{geometry}
\usepackage{amssymb}
\usepackage{amsthm}
\usepackage{amsmath}
\usepackage{graphicx}
\usepackage{todonotes}
\usepackage[all]{xy}
\usepackage{enumitem}

\newcounter{pippo}

\newcommand{\R}{\mathbb R}
\newcommand{\N}{\mathbb N}

\newenvironment{potwr}[1]{{\noindent\it \underline{Proof of Theorem \ref{#1}}}.}{\qed}

\theoremstyle{plain}
\newtheorem{theorem}{Theorem}[section]

\newtheorem{lemma}[theorem]{Lemma}

\theoremstyle{definition}
\newtheorem{definition}[theorem]{Definition}

\theoremstyle{example}
\newtheorem{example}[theorem]{Example}

\numberwithin{equation}{section}

\renewcommand{\Im}{\operatorname{Im}}		
\newcommand{\dist}{\operatorname{dist}}		
\newcommand{\sign}{\operatorname{sign}}	
\newcommand{\sls}[2]{\{#1 \leq #2\} }		
\renewcommand{\sls}[2]{#1^{\leq #2} }		
\renewcommand{\sls}[2]{#1^{(#2)} }		
\newcommand{\spf}{\operatorname{sf}}		
\newcommand{\ps}[2]{$(PS)_{#1,#2}$}		
\newcommand{\partf}{\partial_1 f}	     	
\newcommand{\Zin}{Z_{\operatorname{in}}}
\newcommand{\Zout}{Z_{\operatorname{out}}}
\newcommand{\Sout}{S_{\operatorname{out}}}
\newcommand{\ind}{\operatorname{ind}}		

\newcommand{\J}{\mathcal J}

\usepackage{color} 

\renewcommand{\l}{\lambda}

\begin{document}

\title[Title]{A global bifurcation theorem for critical values of $C^1$ maps in Banach spaces}

\author[P.\ Amster]{Pablo Amster}
\author[P.\ Benevieri]{Pierluigi Benevieri}
\author[J.\ Haddad]{Julian Haddad}

\date{\today}

\address{\smaller Pablo Amster -
Departamento de Matem\'atica, Facultad de Ciencias Exactas y Naturales
Universidad de Buenos Aires and CONICET. 1428 Buenos Aires, Argentina
{\it E-mail address: \tt pamster@dm.uba.ar }}
\address{\smaller Pierluigi Benevieri - Instituto de Matem\'atica e Estat\'istica,
Universidade de S\~ao Paulo,
Rua do Mat\~ao 1010,
S\~ao Paulo - SP - Brasil - CEP 05508-090 -
 {\it E-mail address: \tt
pluigi@ime.usp.br}}
\address{\smaller Juli\'an Haddad - Departamento de Matem\'atica, Instituto de Ci\^encias Exatas, Universidade Federal de Minas Gerais
 {\it E-mail address: \tt
jhaddad@mat.ufmg.br}}

\begin{abstract}
We present a global bifurcation result for critical values of $C^1$ maps in Banach spaces. The approach is topological based on homotopy equivalence of pairs of topological spaces. For $C^2$ maps, we prove a particular global bifurcation result, based on the notion of spectral flow.
\end{abstract}

\maketitle

\section{Introduction}
\label{introduction}

In this paper we present a global bifurcation result for critical values of a $C^1$ map in Banach spaces. 
We proceed in the general spirit of the family of works that uses topological methods, whose origin can be found in the textbook of Krasnoselskij \cite{Kr} in 1964 and the paper of Rabinowitz \cite{Rab} in 1971, even though, we must emphasize, their results concern bifurcation of {\sl solutions} of particular equations, while ours are related to bifurcation of {\sl critical values}, that is, {\sl target values} of a particular function. 

Krasnoselskij obtains the following local bifurcation theorem, which we recall in a simplified version. Let $X$ be a real Banach space. Consider a map $f:\R\times X\to X$ of the form
\[
f(\l,x)=x-\l C(x),
\]
where $C$ is non-linear, compact, Fr\'echet differentiable at $x=0$ and such that $C(0)=0$. We use the term ``compact'' for a continuous map sending bounded subsets of the domain to relatively compact subsets of the target space. The solutions of the equation
\begin{equation}
\label{krasnoeq}
f(\l,x)=0
\end{equation} 
of the form $(\l,0)$ are called \emph{trivial} and a real number $\l_0$ is called a \emph{bifurcation point} of \eqref{krasnoeq} if every neighborhood of $(\l_0,0)$ in $\R\times X$ contains nontrivial solutions.
It is immediate to notice that a necessary condition for $\l_0$ to be a bifurcation point is that the linear operator $I-\l_0C'(0)$ is not invertible, that is, $\l_0$ is a characteristic value of the Fr\'echet derivative $C'(0)$ of $C$ at zero (which is a compact linear operator).

Krasnoselskij proves that $\l_0$ is a bifurcation point of \eqref{krasnoeq} if it is a characteristic value of $C'(0)$ of odd algebraic multiplicity. Rabinowitz extends this result, proving a so called global bifurcation theorem, i.e., showing that 
there exists a connected set $R$ of nontrivial solutions whose closure contains $(\l_0,0)$ and such that at least one of the two alternatives is verified:
\begin{itemize}
\item [i)] $R$ is unbounded,
\item [ii)] the closure of $R$ meets a point of the form $(\l_1,0)$ with $\l_0\neq\l_1$.
\end{itemize} 

It is obvious why Krasnoselskij's result is usually called {\sl local}, while Rabinowitz's one {\sl global}. The approaches of the two authors are based on the application of the Leray--Schauder degree. It is not possible to explain here such a method in details. We limit ourselves to recall the following idea: take $\l\in \R$. If $I-\l C'(0)$ is an automorphism of $X$, we simply denote by the symbol $\deg_{LS}(I-\l C'(0))$ the Leray--Schauder degree of the triple $(I-\l C'(0), U,0)$, where $U$ is any open bounded subset of $X$ containing the origin. Such a value could be $1$ or $-1$, while the Leray--Schauder degree of any triple $(I-\hat\l C'(0),U,0)$ is not defined when $\hat\l$ is a characteristic value of $C'(0)$. The degree is also locally constant, when defined, with respect to $\l$. It can be proven that, when $\l$ crosses a characteristic value $\hat\l$, $\deg_{LS}(I-\l C'(0))$ changes sign if and only if $\hat\l$ has odd algebraic multiplicity. This sign jump is crucial to obtain bifurcation. If, otherwise, the algebraic multiplicity of $\hat\l$ is even, this point could be (or not) a bifurcation point, but the degree does not help to give an answer.

Now, two interesting facts happen:

\begin{itemize}
\item[a)] if, in the equation \eqref{krasnoeq}, $X$ is a real separable Hilbert space and $C'(0)$ is a symmetric (i.e., self-adjoint) operator, then every characteristic value of $C'(0)$ is a bifurcation point;
\item[b)] in some cases, the bifurcation points that are characteristic values of $C'(0)$ of even algebraic multiplicity do not produce a ``global bifurcation branch'' in the sense of Rabinowitz's Theorem.
\end{itemize}

Some questions have been quite naturally stimulated in the last decades and in recent years by the above facts: if one tackles a more 
general 
problem than \eqref{krasnoeq}, is it possible to find a more general degree theory to detect local or global bifurcation? what about more sofisticated topological methods? why do we observe, in some cases, local and not global bifurcation?

More general topological degree theories have been introduced, extending the Leray--Schau\-der degree to compact and non-compact perturbations 
- also multivalued perturbations -
 of nonlinear Fredholm maps between Banach spaces (see, e.g., \cite{BF1,BF4,BCF,ET1,ET2,FPR1,FPR2,RaSa,ZN}). Consequently, local and global bifurcation results have been obtained for more general problems than \eqref{krasnoeq}. We actually have an enormous literature.
%

Consider for example a Banach space $X$ and a $C^1$ map $f:\R\times X\to X$.
Assume $f(\l,0)=0$ for $\l\in \R$. Suppose that, for any $(\l,x)$, the Fr\'echet derivative $\partial_2f(\l,x)$ of $f$ with respect to the second variable at $(\l,x)$ is a Fredholm operator of index zero. 
With a particular notion of orientation 
for Fredholm maps in (possibly infinite dimensional) Banach spaces, it is possible to define a 
topological degree for any partial map $f(\l,\cdot)$
 (see \cite{BF1,FiPeRa}). 
Given $\l\in \R$, denote by $L_{\l}=\partial_2f(\l,0)$.
Suppose $\l_0$ is such that $L_{\l}$ is an isomorphism for $\vert \l-\l_0\vert$ small and nonzero. If the degree of $L_{\l}$ has a sign jump when $\l$ crosses $\l_0$, then $\l_0$ turns out to be a bifurcation point of $f(\l,x)=0$ with a global bifurcation behavior.
Analogously to the case of compact perturbations of the identity studied by Krasnoselskij and Rabinowitz, also in this case the lack of sign jump of the degree does not say anything about bifurcation. 

In the self-adjoint case, the Morse index is a useful tool to detect local bifurcation (but not global, see the above remark b) in some cases for which the degree does not help. Consider a separable real Hilbert space $H$. It is known that, given a self-adjoint Fredholm operator $T:H\to H$, there exists a unique orthogonal splitting of $H$,
\[
H=V^-(T)\oplus V^+(T)\oplus \ker T,
\]
such that $V^-(T)$ and $V^+(T)$ are $T$-invariant, the quadratic form $x\mapsto \langle Tx,x\rangle$ is negative definite on $V^-(T)$ and positive definite on $V^+(T)$.

%
%
With a slight abuse of notation, we will refer to $V^-(T)$ and $V^+(T)$  as the \emph{negative} and the \emph{positive eigenspaces} of $T$, respectively.  The \emph{Morse index} of $T$, denoted by $\mu(T)$, is defined as the dimension of $V^-(T)$ if it is finite.  The following local bifurcation result can be found in the textbook \cite{MW} by Mawhin and Willem (see also \cite{SW} and \cite{FPR99}).  They consider a compact interval $[a,b]$ and an open neighbourhood $U$ of $[a,b]\times \{0\}$ in $\R\times H$.  Given a $C^{2}$ map $\psi:U\to \R$, denote by $L_{\l}$ the Hessian of $\psi_{\l}:=\psi(\l,\cdot)$ at zero, that is, the second derivative of $\psi$ with respect to second variable at the point $(\l,0)$. 

\smallskip

{\bf Theorem A.} \emph{In the above notation, assume that $0\in H$ is a critical point of the functional $\psi_{\l}$ for every $\l\in [a,b]$. In addition, assume that $L_{\l}:H\to H$ is a Fredholm operator and suppose that the negative eigenspace $V^-(L_{\l})$ is finite dimensional for every $\l\in [a,b]$. If 
\[
\mu(L_{a})\neq\mu(L_{b}),
\]
then the interval $[a,b]$ contains a bifurcation point.}

\smallskip

The above important result does not apply in the important case when the operators $L_{\l}$ are so called ``strongly indefinite'', that is, when their positive and negative eigenspaces have infinite dimension.
In order to extend the above result, in a series of papers by Fitzpatrick, Pejsachowicz, Recht, Waterstraat \cite{FPR99,FPR00,PejWat13} a bifurcation problem for a Hamiltonian system is investigated by the application of the spectral flow. The spectral flow has been introduced by Athiyah, Patodi and Singer in \cite{APS} and it is a topologically invariant integer number associated to continuous path of self-adjoint Fredholm operators, $L_\l$, $\l\in [a,b]$, in a separable real Hilbert space $H$.
The spectral flow can be defined by different equivalent methods. In the next section we will summarize its construction, following the approach of Fitzpatrick, Pejsachowicz and Recht \cite{FPR99}. Here, we limit ourselves to observe that
\[
\spf (L,[a,b])=\mu(L_{b})-\mu(l_{a})
\]
when both sides of the above equality are meaningful. 
In \cite{FPR99} it is proven the following extension of Theorem A.

\smallskip

{\bf Theorem B.}
(Fitzpatrick, Pejsachowicz, Recht)
\emph{Let $H$ be a separable real Hilbert space and let $\psi:\R\times H \to \R$ be a $C^2$ function such that, for each $\l\in \R$, $x=0$ is a critical point of the functional $\psi_{\l}:=\psi(\l,\cdot)$. Assume that the Hessian $L_{\l}$ of $\psi_{\l}$ at $0$ is Fredholm and that $L_{a}$ and $L_{b}$ are nonsingular for suitable $a,b$. If $\spf(L, [a, b])\neq 0$, then every neighborhood of $[a, b]\times \{0\}$ contains points $(\l, x)$ such that $x\neq 0$ and is a critical point of $\psi_{\l}$.
}

\smallskip

As previously recalled, the Leray--Schauder degree has been extended in various directions in recent years. Without entering into details, let us just use here a notion of topological degree for a linear Fredholm operator $L:E\to $F between two real Banach spaces that has been introduced by Mawhin (see \cite{CBMS}) and that can be defined for some type of perturbations of linear Fredholm operators between (not necessarily coinciding) Banach spaces. If $T:E\to F$ is an isomorphism between two Banach spaces $E$ and $F$, the degree with respect to $L$ and any open subset of $E$ containing zero, $\deg (T)$, is $\pm1$, depending on a particular concept of orientation induced in the construction (we cannot enter into details here). Coming back to the setting of Theorem B above, one can prove (see \cite{FPR99}) that
\[
(-1)^{\spf(L, [a, b])}=\deg(L_{a})\cdot \deg(L_{b}).
\]

The above equality explains why the spectral flow is a finer invariant than the degree to detect bifurcation, even if it can be applied in a more restricted context. The spectral flow could be nonzero with a lack of sign jump of the degree. In other words, the spectral flow detects bifurcation, if does not vanish, in some cases when the degree does not.

On the other hand, the spectral flow helps to prove local bifurcation results, as in Theorem B, and it seems unable to provide {\sl global} bifurcation results. The reason is probably due to the fact that the spectral flow is defined for {\sl linear operators}, while the degree works in nonlinear maps (see Example \ref{eversion} below). A nonlinear version of the spectral flow could help to obtain global bifurcation results, but to the best of our knowledge it does not exist, and its construction (if possible) is an interesting and challenging open problem in Functional Analysis.

Motivated by these difficulties, in this paper we face a different problem focusing our attention on bifurcation of target values of a suitable function. Our main result, Theorem \ref{mainthm} below, shows the existence of a global bifurcation branch of critical values of a $C^1$ map $f : \R \times X \to \R$, where $X$ is a real Banach space and some topological conditions are verified. This result includes the particular case when $X$ is a separable Hilbert space, $f$ is $C^2$, the Hessians of $f$ with respect to the second variable at the points $(\l,0)$, 
\[
L_\lambda:=\dfrac{\partial^2f}{\partial x^2}(\l,0):H\to H,
\]
are Fredholm and a sufficient condition to obtain bifurcation is given in terms of Morse index (Theorem \ref{mainthm2}). We also obtain a third global bifurcation result, also for the $C^2$ case, when the Hessians $L_\l$ are strongly indefinite and the Morse index is not defined. Adding a particular strong compactness assumption, which seems unremovable, we prove a global bifurcation result if
the spectral flow of $L_\l$ in a suitable interval is nonzero (Theorem \ref{mainthm3}). 

In our first theorem, we obtain the bifurcation result assuming that two suitable topological pairs of inverse images of the map $f_{\l}$ are not homotopically equivalent, for two different values of the parameter $\l$. This condition  is sufficient to give bifurcation when combined with other assumptions (see below), such as a special Palais-Smale type condition.
 
The paper is organized as follows. In Section \ref{preliminaries} we recall some basic notions of homotopic equivalence of topological pairs and we summarize the construction of the spectral flow.
In Section \ref{mresults} we present the bifurcation problem and we state our main results, Theorems \ref{mainthm} and \ref{mainthm2} below. In section \ref{defolemmas} we show some technical results concerning deformation and retraction properties which are used in the proofs of our main results. Such deformations results are original and have in our opinion  some independent interest. Section \ref{proofmainthm} is devoted to the proof of Theorems \ref{mainthm} and \ref{mainthm2}.
Finally, in Section \ref{stronglyind} we provide a bifurcation theorem for strongly indefinite functionals where the corresponding invariant for the Hessian is the spectral flow.

\section{Preliminaries}
\label{preliminaries}

First of all, let us summarize the construction of the spectral flow in the approach followed by Fitzpatrick, Pejsachowicz and Recht in \cite{FPR99}.
Let $H$ be an infinite dimensional separable real Hilbert space. Consider an orthogonal decomposition 
\begin{equation}
\label{pmH}
H=H_{+}\oplus H_{-},
\end{equation}
with $H_{+}$ and $H_{-}$ of infinite dimension. We call \emph{symmetry} the linear operator $\J:H\to H$ which can be represented, in the splitting \eqref{pmH}, by the block-matrix of operators
\[
\left(
\begin{array}{ll}
Id_{H_{+}} & 0\smallskip\\
0 & - Id_{H_{-}}
\end{array}
\right).
\]

Observe that we have infinitely many symmetries of $H$, depending of splittings like the \eqref{pmH}, and that $\J^{2}=Id$ for any symmetry $\J$. Let $\{e_{n}^{{\pm}}, \; n\in \N\}$ be two Hilbert bases of $H_{+}$ and $H_{-}$, respectively, and call $H_{n}$ the $2n$-dimensional subspace of $H$ generated by $\{e_{k}^{{\pm}}, \; k\leq n\}$. Denote by $P_{n}:H\to H_{n}$ the orthogonal projection. Consider a self-adjoint compact linear operator $K:H\to H$ and denote $L=\J+K$. Denote by $L_{n}:H_{n}\to H_{n}$ the operator given by $L_{n}=P_{n}L_{H_{n}}$ and call \emph{signature} of $L_{n}$ the integer number
\[
\sign L_{n}=\mu(-L_{n})-\mu (L_{n}),
\] 
where, as already said, $\mu(\cdot)$ is the Morse index of the considered operator. In \cite[Lemma 1.1]{FPR99}
 the following result is proven:

\medskip

\emph{suppose that the above operator $L=\J+K$ is an automorphism of $H$. Then, there is a positive integer $N$ such that $\sign L_{n}$ is constant if $n\geq N$.}

\medskip

The above eventually constant integer is called \emph{generalized signature} of $L$ with respect to $\J$ and is denoted by $\sign_{\J}(L)$. It is possible to prove that this integer actually depends on the symmetry $\J$ (as the notation suggests), but not on the chosen Hilbert bases $\{e_{n}^{{\pm}}\}$ of the subspaces of $H$ produced by $\J$.

Consider now a continuous path $K_{\l}$, $\l\in [a,b]$ of self-adjoint compact operators of $H$ with $L_{a}$ and $L_{b}$ automorphisms. Given a symmetry $\J$ of $H$, the \emph{spectral flow} of the path $L:[a,b]\to L(H)$, $L_{\l}=\J+K_{\l}$ is defined as
\begin{equation}
\label{sf-formula-1}
\spf (L,[a,b])= \dfrac{\sign_{\J}(L_{b})-\sign_{\J}(L_{a})}{2}
\end{equation}
One can prove that the above formula does not depend on $\J$ even though $\sign_{\J}(L_{a})$ and $\sign_{\J}(L_{b})$ do.

The definition of specral flow can be extended to any continuous path of self-adjoint Fredholm operators $L_{\l}$, $\l\in [a,b]$, such that $L_{a}$ and $L_{b}$ are invertible. 
In the particular case when $L_\l= T + K_\l$, where $T$ is Fredholm and self-adjoint and $K_\l$ is compact, the spectral flow is defined as 
\begin{equation}
\label{sf-formula}
\spf(L, [a, b]) = \dim (V^-(L_a) \cap V^+(L_b)) - \dim (V^-(L_b) \cap V^+(L_a)), 
\end{equation}
which is finite. In the case of a general path of self-adjoint Fredholm operators $L_\l$,
 it can be proven the existence of a path $M_\l$ of automorphisms of $H$ (called \emph{cogredient parametrix}) such that
\[
M_\l^*L_\l M_\l = T + K_\l,
\]
where $M_\l^*$ is the adjoint of $M_\l$, $T$ is Fredholm and self-adjoint and $K_\l$ is compact. 
Hence, the definition $\spf(L, [a, b])$ is given by \eqref{sf-formula} applied to $T + K_\l$ and this does not depend on the choice of $M_\l$.

\medskip

We now recall some basic definitions of homotopy theory. Consider a pair of topological spaces, that is, a pair $(X,A)$ such that $A\subseteq X$. A function between pairs of topological spaces $F:(X,A) \to (Y,B)$ is a continuous function $F:X \to Y$ such that $F(A) \subseteq B$.
Two functions $F,G:(X,A) \to (Y,B)$ are \emph{homotopic} if there exists $H:[0,1] \times X \to Y$ such that $H(0,x) = F(x),\; H(1,x) = G(x)$ and $H(t, x) \in B$ for all $x\in A$ and all $t\in [0,1]$.
Two pairs $(X,A)$ and $(Y, B)$ are \emph{homotopically equivalent} if there exist two functions $F:(X, A) \to (Y,B)$ and $G:(Y,B) \to (X,A)$ such that $G\circ F$ and $F \circ G$ are homotopic to the identity (as functions between pairs).
If this is the case, we have isomorphisms in the relative homology groups $H_i(X,A) \cong H_i(Y,B)$ for every $i$, see for example \cite[pag.\ 118]{H}.
It suffices to consider here singular homology with coefficients in $\R$.

For a continuous function $\varphi:X \to \R$, a critical point $x \in X$ and any $k \in \N$ we consider the  $k$th-local critical group $C_k(\varphi,x)$ (see e.g.\ \cite{KCC,MW} for the definition).
The key ingredient to prove local bifurcation of critical points is the invariance of the critical groups under small perturbations.
The next theorem is proved in \cite[Section 8.9]{MW}. 

\begin{theorem}
Let $U$ be an open neighbourhood of a given point $v$ in a Hilbert space $H$ and 
{consider a map} $\varphi\in C^{2}(U,\R)$
having $v$ as the only critical point and satisfying the Palais-Smale $(PS)$ condition 
over a closed ball $B(v,r) \subseteq U$.

Then, there exists $\eta > 0$, depending only upon $\varphi$, such that for any 
$\psi$ satisfying the same assumptions
the condition
\[\sup_{u \in U}(|\psi(u)-\varphi(u)|) + (|\nabla \psi(u)-\nabla \varphi(u)|) \leq \eta\]
implies
\[\dim C_k(\psi, v) = \dim C_k(\varphi, v), k \in \N.\]
\end{theorem}
As a consequence, if a function $f$ as in Theorem A has non-vanishing spectral flow between $a$ and $b$, then we have 
\[\dim C_k(f_a, 0) \neq \dim C_k(f_b, 0)\]
and thus $0 \in X$ cannot be isolated as a critical point for every $\lambda$.

We will consider a global version of the local critical groups explained by Theorem \ref{mainprelimthm} below, that can be found in \cite[Theorem 5.1.27]{KCC}. We need first the following definition. 
\begin{definition}
\label{def_sublevelset}
Let $\varphi:X \to \R$ be any continuous function, denote 
\[
\sls{\varphi}{a}=\{x \in X/ \varphi(x) \leq a\}.
\]

\end{definition}

\begin{theorem}
\label{mainprelimthm}
Assume that $\varphi: X \to \R$ is $C^1$ and satisfies the $(PS)$ condition. Suppose that $c$ is an isolated critical value of $\varphi$, where the critical points of $\varphi$ in $\varphi^{-1}(c)$ are $z_1, \ldots, z_m$.
Then for sufficiently small $\varepsilon > 0$ we have
\[H_k(\sls \varphi {c+\varepsilon}, \sls \varphi {c-\varepsilon}) = \bigoplus_{j=1}^m C_k(\varphi, z_j)\]
for every $k = 0,1,2, \ldots$
\end{theorem}

Theorem \ref{mainprelimthm} will play a central role in the proof of our main results.

\section{Main Results}
\label{mresults}

{\bf Standing assumption.} Let $X$ be a real Banach space. In what follows (and unless otherwise explicitely stated) $f :\R \times X \to \R$ will stand for a $C^1$ map such that 
\begin{equation}
\label{zero-image}
f(\lambda,0) = 0 \quad \textrm{and} \quad\frac{\partial f}{\partial x} (\lambda,0) =f'_\lambda(0)= 0, \quad \forall \lambda\in \R.
\end{equation}

The map $x\in X\mapsto f(\lambda,x)$, defined for a given real $\lambda$, will be also denoted by $f_{\lambda}$ and its Fr\'echet derivative at a point $x$ by $f'_\lambda(x)$.
The following set 
\[
S_f = \{(\lambda, y) \in \R^2 / f_\lambda(x) = y,\  f'_\lambda(x) = 0 \hbox{ for some } x \in X\}
\]
will be called the set of \emph{critical pairs}.
In other words, $S_f$ is the set of pairs $(\lambda, y)$ such that $y$ is a critical value of $f_\lambda$. The line $Z = \R \times \{0\} \subseteq S_f$ is regarded as the set of the \emph{trivial critical pairs}.
We say that a trivial critical pair $(\l_{0},0) \in Z$ is a \emph{bifurcation point} if every neighbourhood of $(\l_{0},0)$ contains non-trivial critical pairs.
Then, we see that 
\begin{equation}
\label{setE}
E_f := \overline{S_f \setminus Z} 
\end{equation}
is the union of the bifurcation points and the non-trivial critical pairs.
We will also split, when necessary, the sets of trivial critical pairs into the following two subsets:
\[\Zin  = [-1,1] \times \{0\} \subset \R^2, \quad \Zout = \left( (-\infty, -1] \cup [1,\infty) \right) \times \{0\} \subset \R^2.\]

\begin{definition}
[(PS)-type conditions]\label{def_PSI}
Let $I \subset \R$ be a compact interval and $c\in \R$ a given value. We say that $f$ satisfies the \ps{I}{c} condition if, for every sequence $(\lambda_n, x_n) \in I \times X$ such that $f_{\lambda_n}(x_n) \to c$ and such that $f_{\lambda_n}'(x_n) \to 0 \in X^*$, there exists a sub-sequence $(x_{n_k})$ converging to a point  $ x_0 \in X$.

Let $J \;  \subset  \; \R$, we say that $f$ satisfies the \ps{I}{J} condition if it satisfies the \ps{I}{c} condition for every $c \in J$.
We say that $f$ satisfies the \ps{\R}{\R} condition if it satisfies the \ps{[-N,N]}{\R} condition for every $N \in \N$.
\end{definition}

We are now in a position to state the following two theorems which, associated with Theorem \ref{mainthm3}, are our main results. Actually, Theorem \ref{mainthm2} is a consequence of Theorem \ref{mainthm}, particularly important in applications. The proof of both results will be given in Section \ref{proofmainthm}.

\begin{theorem}
\label{mainthm}
Let $f : \R \times X \to \R$ be a $C^1$ function verifying the standing assumption \eqref{zero-image} above.  Denote by $\partf: \R \times X \to \R$ the map
\[
\partf(\lambda, x) = \frac{\partial f}{\partial \lambda}(\lambda, x).
\]
Denote by $I$ the interval $ [-1,1]$ and assume that  the following assumptions hold: 
\begin{enumerate}[label={\normalfont \roman*\_}, ref = \ref{mainthm}-\roman*]
\item \label{mainthm_ps} $f$ satisfies the \ps{\R}{\R} condition.
\item \label{mainthm_partf_bounded} $\partf$ is bounded in the sets of the form 
\[
f_{[-N,N]}^{-1}([-N,N]):=\{(\l,x)\in [-N,N] \times X: f(\l,x)\in [-N,N]\},  \quad \forall N \in \N.
\]
\item \label{mainthm_nonbif} The trivial critical pairs $(-1,0), (1,0)$ are not bifurcation points of $f$.
\item \label{mainthm_specflow} There exists $\varepsilon_* > 0$ such that for every $0 < \varepsilon < \varepsilon_*$ the 
pairs of spaces $\left(\sls{f_{-1}}{\varepsilon}, \sls{f_{-1}}{-\varepsilon}\right)$ and $\left(\sls{f_1}{\varepsilon}, \sls{f_1}{-\varepsilon}\right)$ are not homotopically equivalent.
\end{enumerate}
Then, $E_f$ contains a connected subset intersecting $\Zin $ which either
\begin{enumerate}[label={\normalfont \arabic*\_}, ref = \ref{mainthm}-\arabic*]
\item \label{mainthm_unbounded} is unbounded in $\R^2$, or else
\item \label{mainthm_intersects} intersects $\{-1,1\} \times X$.
\end{enumerate}
%
\end{theorem}


\begin{theorem}
\label{mainthm2}
Let $H$ be a separable real Hilbert space and consider a $C^2$ function $f : \R \times H \to \R$. Assume that
\[
f(\lambda,0) = 0\quad \textrm{and} \quad  \nabla f_\lambda( 0) = 0,  \quad \quad \forall \lambda \in \R.
\] 
 
Denote by $I$ the interval $ [-1,1]$ and assume that  the following assumptions hold: 

\begin{enumerate}[label={\normalfont \roman*\_}, ref = \ref{mainthm2}-\roman*]
\item \label{mainthm2_ps} $f$ satisfies the \ps{\R}{\R} condition.
\item \label{mainthm2_partf_bounded} $\partf$ is bounded in the sets of the form $f_{[-N, N]}^{-1}([-N, N])$ for every $N \in \N$.
\item \label{mainthm2_nondeg} For $i = -1,1$, the point $0 \in H$ is a non-degenerate critical point and the only critical point of $f_i$ with value $0$. 
\item \label{mainthm2_specflow} Assume that, for every $\lambda\in I$, the Hessian of $f$ at zero, 
\[
L_\lambda:=\dfrac{\partial^2f}{\partial x^2}(\l,0):H\to H
\]
is Fredholm for every $\l\in [-1,1]$ and suppose
\[
\mu(L_{-1}) \neq \mu(L_1).
\]
where $\mu$ denotes the Morse index.

\end{enumerate}
Then the conclusion of Theorem \ref{mainthm} holds.
\end{theorem}

One may ask if Theorem \ref{mainthm} works with critical points instead of critical values, as in the bifurcation theorems of Rabinowitz \cite{Rab}.
The following example shows otherwise.
\begin{example}
\label{eversion}
\normalfont
A classic problem in analysis consists in proving the existence of a {\it cone eversion}.
This is a smooth
 function $c:I \times C \to \R$ where 
\[C=\{x \in \R^2/ 1\leq \|x\| \leq 2\},\] such that 
\[c(-1,x) = {\|x\|}, \ \ c(1,x) = - {\|x\|}\]
and such that $c_\lambda$ has no critical points in $C$ for $\lambda \in I$.
The existence of such a function may seem counterintuitive but actually is guaranteed by the Parametric Holonomic Approximation Theorem, see Example 4.1.1 \cite{EM}.
An explicit formula for $c$ was computed in \cite{T}.
We construct $f :\R \times X \to X$ with $X = \R^2$ as follows.

The formula for $c$ in polar coordinates as given in \cite{T} is
\[c(\lambda, (\alpha, r)) = 2t + g(\lambda, \alpha) + (r-2) h(\lambda, \alpha)\]
for two functions $g,h$ satisfying 
\[\left( \frac{\partial g}{\partial \alpha}(\lambda, \alpha), h(\lambda, \alpha)\right) \neq (0,0)\]
for all $(\lambda, \alpha)$.

Consider 
$p_{a,b}(r) = (3 a-b+1)r^2 + (-2 a+b-2)r^3 +r^4$
, which is the polynomial function satisfying the properties
\[
p_{a,b}(0)=p_{a,b}'(0) = 0, \ p_{a,b}(1)=a,\ p'_{a,b}(1)=b,
\]
and define
\[f(\lambda, (\alpha, r) ) = p_{g(\lambda, \alpha), 3 h(\lambda, \alpha)}( r ).\]

We easily verify
\begin{enumerate}[label={\normalfont \roman*}, ref = \roman*]
\item \label{example_g_ps} $f(\lambda, x), \nabla f(\lambda,x) \to +\infty$ uniformly in $\lambda$, as $x \to \infty$.
\item \label{example_g_crit} $f(\lambda, 0) = 0, \ \ \nabla f(\lambda, 0) = 0$  for all $\lambda \in I$
\item \label{example_g_spf} $f(-1,(\alpha, r) ) = 4 r^2 - 3 r^3 + r^4 \\ f(1,(\alpha, r)) = -2 r^2 - r^3 + r^4$ 
\item \label{example_g_nocrit} $\nabla f(\lambda,x) \neq 0$ for $\|x\| = 1$, $\lambda \in I$.
\end{enumerate}

\begin{center}
\includegraphics[scale=.7]{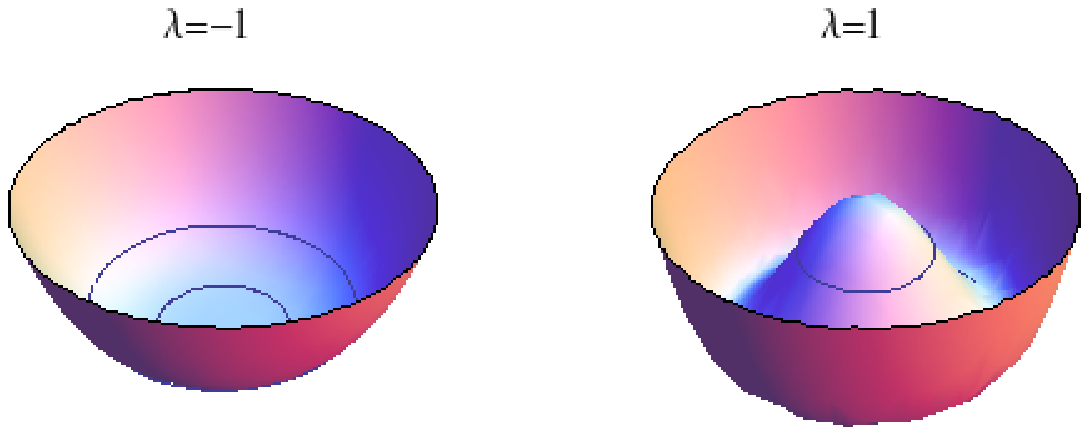}
\end{center}
Property \eqref{example_g_spf} implies $\spf(f'(.,0),I) = -2$, and in view of the relation between the spectral flow and the topological index, we have
\[\ind(\nabla f_{-1}, 0) . \ind(\nabla f_{1}, 0) = (-1)^{\spf(f'(.,0),I)} = 1\]
and $\nabla f$ does not satisfy the necessary hypothesis for the ``classical'' bifurcation theorem.
Also, one may check that $f$ is in the conditions of Theorem \ref{mainthm}.

Let 
\[K_f = \{(\lambda, x) \in I \times X/ \nabla f_\lambda(x) = 0\}\]
 and notice that condition \eqref{example_g_nocrit} forces $K_f \cap (I \times \partial B(0,1)) = \emptyset$.
Then, the connected component of $\overline{K_f \setminus (I \times \{0\})}$ containing bifurcation points $(\lambda, 0)$ is inside $I \times B(0,1)$ and does not intersect the subspaces $\lambda = -1,1$.
This is, there is {\it local} but not {\it global} bifurcation.

The function $g_1$ has a set of non-trivial critical points in $X$ of the form $\partial B(0,R)$ with $R \in (1,2)$.
Actually the connected component of $K_f$ containing $\{1\}\times \partial B(0,R)$ lies outside $I \times B(0,1)$ but its image by $g$ crosses the value $0$.
\end{example}

This example shows that it is not possible to prove {\it global} bifurcation of {\it critical points} of functions having non-vanishing spectral flow.

The proofs of Theorems \ref{mainthm} and \ref{mainthm2} require several lemmas.
We start with some results concerning the level sets and the \ps{I}{c} condition.

\section{Deformation Theorems}
\label{defolemmas}

We present here some technical lemmas which will play a fundamental role in the proof of Theorems \ref{mainthm} and \ref{mainthm2}. The next lemma is a sort of characterization of the \ps{I}{c} property. In this section, $I$ denotes a real compact interval and the function $f$ is not required to verify the assumption \eqref{zero-image}.

\begin{lemma}
\label{lemma_boundgradient}
Let $f: \R \times X \to \R$ be a $C^1$ function satisfying the \ps{I}{a} condition for a given $a\in \R$ which is regular value for every $f_\lambda$, with $\l\in I$.
Then, there exist $\varepsilon, \delta > 0$ such that
$\|f_\lambda'(x)\| \geq \varepsilon$ for every $(\lambda, x) \in f_{I}^{-1}((a - \delta, a + \delta) )$.
\end{lemma}
\begin{proof}
Assume by contradiction that there exists a sequence $(\lambda_n, x_n) \in I \times X$ such that $\|f_{\lambda_n}'(x_n)\| \to 0$ and $f(\lambda_n,x_n) \to a$.
Taking a convergent sub-sequence $(\lambda_{n_j})$ and a convergent sub-sequence $(x_{n_j})$ given by the \ps{I}{a} condition, 
we deduce that $a$ is a singular value.
\end{proof}

\begin{lemma}
\label{lemma_Kclosed}
Consider a closed set $J \subset \R$ and let $f: \R \times X \to \R$ be a $C^1$ function satisfying the \ps{I}{J} condition.
Then the set 
\[
K = \{(\lambda, c) \in I \times J/ \exists x \in X, f_\lambda(x) = c, f_\lambda'(x) = 0\}
\] is closed.
\end{lemma}
\begin{proof}
Take a convergent sequence $(\lambda_n, c_n) \to (\lambda, c) \in I \times J, \ (\lambda_n, c_n) \in K$, and consider  $x_n \in X$, for any $n$, such that $f(\lambda_n, x_n) = c_n$, $f_\lambda'(x_n) = 0$.
By the \ps{I}{c} condition there exists a convergent sub-sequence of $(x_n)$. Thus, the continuity of $f_\lambda'$ shows that $(\lambda, c) \in K$.
\end{proof}

\begin{lemma}
\label{lemma_vectorfield}
Let $f: \R \times X \to \R$ be a $C^1$ function satisfying \ps{I}{\{a,b\}} where $a<b$ are regular values of $f_\lambda$ for $\lambda \in I$.
Denote $U_t = (a-t, a+t) \cup (b-t, b+t)$.
Assume $\partf$ is bounded in $f_{I }^{-1}(U_\delta)$ for some $\delta > 0$.
Then there exists a function $v: \R \times X \to X$ such that
\begin{enumerate}[label={\normalfont \roman*\_}, ref=\roman*]
\item $v$ is bounded and locally Lipschitz.
\item \label{lemma_vectorfield_decre} $\partf(\lambda, x) + f_\lambda'(x)[v(\lambda, x)] < 0$ if $\lambda \in I$ and $f_\lambda(x) = a \hbox{ or } b$.
\end{enumerate}
\end{lemma}

The reader can understand the notation $f_{I }^{-1}(U_\delta)$ by the analogous set in the statement of  Theorem \ref{mainthm}.

\begin{proof}
First, by Lemma \ref{lemma_boundgradient} we may assume, taking a smaller $\delta>0$ if necessary, that, for some $\varepsilon>0$, $\|f_\lambda'(x)\| \geq \varepsilon$ for all $(\lambda, x) \in f_{I }^{-1}(U_\delta)$.
Let $\nu:\R \to [0,1]$ be a continuous  function equal to $0$ in $U_{\delta/2}$ and equal to $1$ in $\R \setminus U_{\delta}$.
We shall construct $v(\lambda, x)$ satisfying the  inequality
\begin{equation}
\label{lemma_vectorfield_ineq}
\partf(\lambda, x) + f_\lambda'(x)[v(\lambda, x)] < \nu(f_\lambda(x))\left(|\partf(\lambda, x)| +1 \right), \quad \forall\lambda \in I.
\end{equation}
Thus, condition \eqref{lemma_vectorfield_decre} will follow. To this purpose, fix $(\lambda, x) \in I \times X$ and first assume $f_\lambda(x) \in U_\delta$, so that $\|f_\lambda'(x)\| \geq \varepsilon$.
There exists $w_{\lambda, x}\in X$ such that 

\[
\|w_{\lambda, x}\| \leq \frac 2 \varepsilon \left(|\partf(\lambda, x)| + 1\right)
\]

and 
\[
f_\lambda'(x)[w_{\lambda, x}] \leq -(|\partf(\lambda, x)| + 1).
\]

Hence, we have $\partf(\lambda, x) + f_\lambda'(x)[w_{\lambda, x}] < 0$.
On the other hand, if $f_\lambda(x) \not\in U_\delta$ we define $w_{\lambda, x} = 0$. 
%
By continuity,  every $(\lambda, x) \in I \times X$ has  a neighbourhood $V^{\lambda, x}$ in  $\R \times X$ such that
\begin{equation}
\label{ineq-1}
\partf(\alpha, y) + f_\alpha'(y)[w_{\lambda, x}] <0< \nu(f_\alpha(y))\left(|\partf(\alpha, y)| +1 \right)
\end{equation}

for every $(\alpha, y) \in V^{\lambda, x}$.
Since $I \times X$ is paracompact 
 and $\{V^{\lambda, x} \}_{\lambda, x}$ is an open covering of $I \times X$, we obtain a countable, locally finite refinement $V_i \subseteq V^{\lambda_i, x_i}$ covering $I \times X$, and a locally Lipschitz partition of unity, this is, a collection of (locally Lipschitz) non-negative functions $\eta_i:\R \times X \to \R$ with support in $V_i$ and such that 

\[\sum_{i\in \N} \eta_i (\l,x)= 1  \quad \quad \forall (\l,x)\in I \times X,
\]

where the above sum is locally finite. Recalling \eqref{ineq-1}, we have,  for any $i$ and any $(\alpha, y)\in I \times X$,
\[
\eta_i(\alpha, y) \partf(\alpha, y) + f_\alpha'(y)[\eta_i(\alpha, y) w_{\lambda, x}] \leq \eta_i(\alpha, y) \nu(f_\alpha(y))\left(|\partf(\alpha, y)| +1 \right),
\]
with strict inequality if $\eta_i(\alpha, y) > 0$. Then we obtain
\begin{align*}
\sum_{i \in \N} \eta_i(\alpha, y) \left( \partf(\alpha, y) + f_\alpha'(y)[w_{\lambda, x}] \right) &< \sum_{i \in \N} \eta_i(\alpha, y) \nu(f_\alpha(y))\left(|\partf(\alpha, y)| +1 \right).
\end{align*}

Now, define
\[v(\lambda, x) =\sum_{i \in \N} \eta_i(\lambda, x) w_{\lambda_i, x_i}.
\]

We have
\begin{align*}
\partf(\alpha, y) + f_\alpha'(y)[v(\lambda, x)] &< \nu(f_\alpha(y))\left(|\partf(\alpha, y)| +1 \right).
\end{align*}

Since $\|w_{\lambda, x}\| \leq \frac 2 \varepsilon \left(|\partf(\lambda, x)| + 1\right)$ for every $(\lambda, x) \in I \times X$, we have, for any $(\lambda, x) \in \R \times X$,
\[ \|v(\lambda, x)\| \leq  \sum_{i \in \N} \eta_i(\lambda, x) \|w_{\lambda_i, x_i} \| \leq \frac 2\varepsilon \left( \sup_{(\lambda, x) \in f_{I }^{-1}(U_\delta)} |\partf(\lambda, x)| + 1 \right) \]
and $v$ is bounded. Finally, it is immediate to observe that $v$ is locally Lipschitz and this concludes the proof.
\end{proof}       

The next three theorems are key to the study of deformations of one-parameter families of functions.
Theorems \ref{thm_deformation1} and \ref{thm_deformation2} below generalize analogous results in the textbook \cite{KCC}, where they appear under stronger conditions.

\begin{theorem}[Deformation Theorem]
\label{thm_deformation1}
Let $f: \R \times X \to \R$ be $C^1$ and $a<b$ regular values of $f_\lambda$ for $\lambda \in I$.
Assume that $f$ satisfies \ps{I}{\{a,b\}} and that $\partf$ is bounded in $f_{I }^{-1}(U_\delta)$ for some $\delta > 0$.
Then the pairs $\left(\sls{f_{-1}}{b}, \sls{f_{-1}}{a}\right)$ and $\left(\sls{f_1}{b}, \sls{f_1}{a}\right)$ are homotopically equivalent.
\end{theorem}
\begin{proof}
The vector field $v$ constructed in Lemma \ref{lemma_vectorfield} is locally Lipschitz and bounded, so it generates a globally defined flow
\[
\phi: \R \times \R \times X \to X,
\]
which satisfies
\[
\frac{\partial}{\partial t} \phi(t, \lambda, x) = v(t+\lambda, \phi(t, \lambda, x)).
\]
For $(\lambda_0, x_0) \in I \times X$, the curve $x(t) = \phi(t, \lambda_0, x_0)$ is the solution of the initial value problem
\begin{equation}
\left\{
\begin{array}{rl}
x'(t) &= v(\lambda_0 + t,x(t))\\
x(\lambda_0) &= x_0.
\end{array}
\right.
\end{equation}

For any $(\lambda, x) \in \R \times X$ define $\varphi(t) = f_{t+\lambda}(\phi(t, \lambda, x))$, which verifies
\[ \varphi'(t) = \partf(t+\lambda, \phi(t, \lambda, x)) + f_{t+\lambda}'(\phi(t, \lambda, x))[v(t+\lambda, \phi(t, \lambda, x)]. \]

Notice that, by condition \eqref{lemma_vectorfield_decre} of Lemma \ref{lemma_vectorfield}, $\varphi(t) = a$ implies $\varphi'(t) < 0$. So, if $\varphi(0) \leq a$, then we have $\varphi(t) < a$ for every $t>0$.
If $x \in \sls{f_\lambda}{a}$, then, by the definition of $\varphi$ and the previous consideration, $\phi(t, \lambda, x) \in \sls{f_{t+\lambda}}{a}$.
We write this is as
\[\phi\left(t,\lambda,\sls{f_\lambda}{a}\right) \subseteq \sls{f_{t+\lambda}}{a}.\]
Similarly, for $b$ we obtain 
\[
\phi\left(t,\lambda,\sls{f_\lambda}{b}\right) \subseteq \sls{f_{t+\lambda}}{b}.
\]

Appliying the same reasoning to the function $\bar f(\lambda, x) := f(-\lambda, x)$, we obtain a (globally defined) flow $\bar\phi: \R \times \R \times X \to X$ such that
\[\bar\phi\left(t,\lambda,\sls{\bar f_\lambda}{a}\right) \subseteq \sls{\bar f_{t+\lambda}}{a},\ \  \bar\phi\left(t,\lambda,\sls{\bar f_\lambda}{b}\right) \subseteq \sls{\bar f_{t+\lambda}}{b}\]
for every $t>0$.

The functions $F(x) = \phi(2,-1,x), G(x) = \bar\phi(2,-1,x)$ are continuous functions of pairs
\[F:\left(\sls{f_{-1}}{b}, \sls{f_{-1}}{a}\right) \to \left(\sls{f_1}{b}, \sls{f_1}{a}\right)\]
\[G:\left(\sls{f_1}{b}, \sls{f_1}{a}\right) \to \left(\sls{f_{-1}}{b}, \sls{f_{-1}}{a}\right).\]

Define 
\[H_t(x) = \bar\phi(t,1-t,\phi(t,-1,x)),\ \  \bar H_t(x) = \phi(t,1-t,\bar\phi(t,-1,x)).
\]
Then, we verify that $H_t, \bar H_t$ are functions of pairs
\[H_t:\left(\sls{f_{-1}}{b}, \sls{f_{-1}}{a}\right) \to \left(\sls{f_{-1}}{b}, \sls{f_{-1}}{a}\right),\ \ 
\bar H_t:\left(\sls{f_1}{b}, \sls{f_1}{a}\right) \to \left(\sls{f_1}{b}, \sls{f_1}{a}\right)\]
for $t>0$. In addition, $H_0(x) = x, H_2(x) = G(F(x)), \bar H_0(x) = x, \bar H_2(x) = F(G(x))$.
Thus, $F$ and $G$ are homotopy equivalences.
\end{proof}

\begin{theorem}
\label{thm_deformation2}
Let $(r,s)$ be an open bounded interval and let $a,b:I \to (r,s)$ be two continuous functions such that $a(\lambda) < b(\lambda)$ for any $\lambda$, and $f: \R \times X \to \R$ a $C^1$ function satisfying the \ps{I}{[r,s]}-condition.

Assume that $\partf: I \times X \to \R$ is bounded in the set $f_{I }^{-1}(r, s)$.
Assume also that $a(\lambda), b(\lambda)$ are regular values of $f_\lambda$ for every $\lambda \in I$.
Then the pairs $\left(\sls {f_{-1}}{b({-1})}, \sls {f_{-1}}{a({-1})}\right)$ and $\left(\sls {f_1}{b(1)}, \sls {f_1}{a(1)}\right)$
are homotopically equivalent.
\end{theorem}
\begin{proof}
By hypothesis, the graphs of $a,b$ in $I \times \R$ do not intersect the set 
\[
K = \{(\lambda, c) \in I \times [r,s]: \; \exists x \in X, f_\lambda(x) = c, f_\lambda'(x) = 0\}.
\]

Since $K$ is closed by Lemma \ref{lemma_Kclosed}, the graphs of the functions $a,b$ can be approximated by $C^1$ functions with the same endpoints $a(\pm 1), b(\pm 1)$. Thus we may assume $a,b$ are $C^1$.
We consider the function
\[g(\lambda, x) = \frac{ f(\lambda, x) - a(\lambda)}{b(\lambda) - a(\lambda)}.\]
Thus, $g$ satisfies 
\[
\sls{f_\lambda}{a(\lambda)} = \sls{g_\lambda}{0},\quad \quad \sls{f_\lambda}{b(\lambda)} = \sls{g_\lambda}{1},
\]
and it is easy to check that $g$ verifies the conditions of Theorem \ref{thm_deformation1} with $a=0, b=1$.
\end{proof}

%
%
\section{Proofs of Theorems \ref{mainthm} and \ref{mainthm2} }
\label{proofmainthm}
First we prove two technical lemmas.

\begin{lemma}
\label{lemma_square}
Let $D \subset \R^2$ be a closed rectangle, let $A_0, A_1$ be two opposite sides of $D$ and $B_0, B_1$ the other opposite ones.
Let $S \subset D$ be a compact set.
If $S$ does not contain a connected component intersecting $B_0$ and $B_1$ then there exists a continuous curve $\gamma:[0,1] \to D$ with $\gamma(0) \in A_0$ and $\gamma(1) \in A_1$ and not intersecting $S$, $B_0$ and $B_1$.
\end{lemma}
\begin{proof}
Without loss of generality, we can assume that $A_0, A_1$ are vertical sides and $B_0, B_1$ horizontal. Let $S_i = S \cap B_i, i = 0,1$ and assume there is no connected component of $S$ which intersects $S_0$ and $S_1$ at the same time.
We may assume that $S_i \neq \emptyset$ because, otherwise,
a horizontal line close to $B_i$ would be the desired curve.
By Whyburn's Lemma there is a separation $S = C_0 \cup C_1$ where $S_i \subseteq C_i$ and $C_0, C_1$ are two non-empty disjoint compact sets.
By the smooth Urysohn Lemma (\cite[corollary of Theorem 1.11]{WA}) there exists a $C^\infty$ smooth function $g: \R^2 \to \R$ such that $g(x) = i$ for all $x \in C_i, i = 0,1$.

Take $\alpha \in (0,1)$ a regular value of the three functions $g, g|_{A_0}, g|_{A_1}$ and consider $L = g^{-1}(\{\alpha\}) \cap D \setminus B_0 \setminus B_1$ which is a differentiable manifold with boundary, of dimension $1$ with $\partial L = L \cap (A_0 \cup A_1)$ and $L \cap K = \emptyset$.
Since $g(x) = i$ for all $x \in B_i, i=0,1$ and $\alpha$ is a regular value of $g|_{A_i}, i=0,1$, then we have that the cardinality of $L \cap A_i$ is odd, for $i=0,1$. Therefore, there is a curve of $L$ having one endpoint in each of the $A_i$'s.
\end{proof}

\begin{lemma}
\label{lemma_two_curves}
Let $a, b :[-1,1] \to [-1,1]$ be two continuous functions satisfying 

\[
a(-1)=b(-1)=-1, \quad \quad a(1)=b(1)=1.
\]

Then for every $\varepsilon > 0$ there exist continuous functions $\tilde a, \tilde b, c, d:[-1,1] \to [-1,1]$ such that
\begin{enumerate}
\item $\tilde a(-1)=\tilde b(-1)=-1$, \quad $\tilde a(1)=\tilde b(1)=1$,
\item $c(-1)=d(-1)=-1$, \quad $c(1)=d(1)=1$,
\item $\|\tilde a - a\|_\infty < \varepsilon, \quad  \|\tilde b - b\|_\infty < \varepsilon$,
\item $\tilde a(c(t)) = \tilde b(d(t))$.
\end{enumerate}
\end{lemma}
\begin{proof}
We approximate $a,b$ by smooth functions $\tilde a, \tilde b :[-1,1] \to [-1,1]$ satisfying conditions $1,3$ above and
\begin{enumerate}[label=\roman*, ref=\roman*]
\item \label{lemma_two_curves_i} $\tilde a'(-1), \tilde b'(-1), \tilde a'(1), \tilde b'(1) > 0$,
\item \label{lemma_two_curves_ii} the critical values of $\tilde a$ and $\tilde b$ are disjoint,
\item \label{lemma_two_curves_iii} $\tilde a(x), \tilde b(x) \in (-1,1)$ for all $x \in (-1,1)$.
\end{enumerate}

Let $D = [-1,1] \times [-1,1]$ and $\phi:D  \to \R$, defined as  $\phi(x,y) = \tilde a(x) - \tilde b(y)$.
Condition \eqref{lemma_two_curves_ii} guarantees that $0$ is a regular value of $\phi|_{D^0}$.
Let $L = \phi^{-1}(0)$. Since condition \eqref{lemma_two_curves_iii} implies 
\[
\begin{array}{lll}
\phi(1,t)>0, & &  \phi(t,-1) >0,\\
\phi(-1,t)<0, & & \phi(t,1) < 0
\end{array}
\]

for all $t \in (-1,1)$, we have $L \cap \partial D= \{(-1,-1), (1,1)\}$.
Also, by condition \eqref{lemma_two_curves_i} , $L$ can be parametrized near $(-1,-1)$ and $(1,1)$ with curves entering $D^0$.
Thus $L$ is a differentiable manifold of dimension $1$ with boundary $\{(-1,-1), (1,1)\}$ which must connect these two points.
Parametrize the curve inside $L$ connecting $(-1,-1)$ and $(1,1)$ by $(c(t), d(t))$.
Then $c, d$ satisfy 

\[
0 = \phi(c(t), d(t)) = \tilde a(c(t)) - \tilde b(d(t)).
\]
\end{proof}

\begin{figure}
\includegraphics{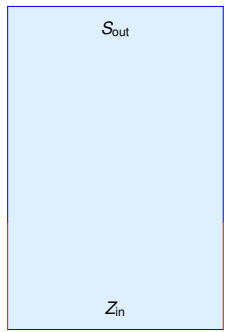}
\caption{The red sides correspond to the vertical sides of the square and the blue sides, to the horizontal ones.}
\end{figure}

\medskip
\medskip
\medskip

\begin{potwr}{mainthm}
Let $\varepsilon > 0$ be small enough so that the two disks $B_{-1} = B((-1,0), \varepsilon)$ and $B_1 = B((1,0), \varepsilon)$ contain only trivial critical pairs. 

Let $G_f$ be the connected component of $E_f \cup \Zin$ containing $\Zin$.
Assume by contradiction that none of the alternatives \eqref{mainthm_unbounded}, \eqref{mainthm_intersects} are satisfied for $G_f$, then there is $R>0$ such that $G_f \subseteq [-1,1] \times [-R,R]$.
The set $D_+ = [-1,1] \times [0,R]$ is homeomorphic to a closed rectangle, where the vertical sides correspond to $\{\pm 1\} \times [0,\varepsilon]$ and the horizontal lines to the rest of $\partial D_+$, this is
$\Zin$ and the three segments $\Sout = ( \{-1,1\} \times [\varepsilon,R] ) \cup ( [-1,1] \times \{R\} )$.
The reader can see the above picture.
If condition \eqref{mainthm_intersects} fails then there is no connected set inside $E_f$ that intersects at the same time both horizontal sides of the rectangle.
Therefore, applying Lemma \ref{lemma_square} which is invariant by homeomorphisms, as it is easy to see, we conclude that there exists a continuous curve 
\[
\widehat u_+:[-1,1] \to \overline D_+ \setminus E_f, 
\]
such that 
\[
\widehat u_+(-1) = (-1,\varepsilon/2), \quad \widehat u_+(1) = (1,\varepsilon/2).
\]

Similarly, define $D_-$ in the lower semiplane and a map 
\[
u_-:[-1,1] \to D_- \setminus E_f,
\]
such that 
\[
u_-(-1) = (-1,-\varepsilon/2), \quad u_-(1) = (1,-\varepsilon/2),
\]       
and having analogous properties to $u_{+}.$
Denote
\[
u_{+}(t)=(\l_{+}(t),y_{+}(t)), 
u_{-}(t)=(\l_{-}(t),y_{-}(t)). 
\]

By Lemma \ref{lemma_two_curves} with 
$\varepsilon < \dist(S_f, \Im(u_\pm))$ and 
$a,b$ replaced by $\lambda_- , \lambda_+$, we obtain functions $\tilde \lambda_-, \tilde \lambda_+, c, d:[-1,1] \to [-1,1]$ such that
\begin{enumerate}
\item $\tilde u_\pm(t): = (\tilde \lambda_\pm(t), y_\pm(t) ) \in \R^2 \setminus S_f$ for all $t \in [-1,1]$,
\item $\tilde\lambda_-(c(t)) = \tilde\lambda_+(d(t))$.
\end{enumerate}
Now define 
\begin{align*}
\lambda(t) & = \tilde\lambda_-(c(t)) = \tilde\lambda_+(d(t)),\\
a(t) & = y_-(c(t)),\\
b(t) & = y_+(d(t)).
\end{align*}

By the properties of the curves $\tilde u_\pm$ we know that $a(t)$ and $b(t)$ are regular values of $f_{\lambda(t)}:X\to \R$, for any $t\in [0,1]$.
Applying Theorem \ref{thm_deformation2} to the family of maps $f_{\lambda(t)}$, we get the homotopy equivalence of pairs
\[\left(\sls{f_{-1}}{b({-1})}, \sls{f_{-1}}{a({-1})}\right) \cong \left(\sls{f_1}{b(1)}, \sls{f_1}{a(1)}\right).\]
In view of condition \eqref{mainthm_nonbif}, we have
\[\left(\sls{f_{-1}}{\delta}, \sls{f_{-1}}{-\delta}\right) \cong \left(\sls{f_1}{\delta}, \sls{f_1}{-\delta}\right)\]
for any $\delta \in (0, \varepsilon/2)$, which contradicts hypothesis \eqref{mainthm_specflow}.

\end{potwr}

\medskip

\begin{potwr}{mainthm2}
We must prove that conditions \ref{mainthm2_ps} through \ref{mainthm2_specflow} imply the conditions \ref{mainthm_ps} through \ref{mainthm2_specflow}.

The facts that $f$ is $C^2$ and that $0 \in H$ is a non-degenerate critical point (condition \ref{mainthm2_nondeg}) imply \ref{mainthm_nonbif}.

It remains to show that conditions \ref{mainthm2_nondeg} and \ref{mainthm2_specflow} imply condition \ref{mainthm_specflow}. 
Assume the pairs of spaces $\left(\sls{f_{-1}}{\delta}, \sls{f_{-1}}{-\delta}\right)$ and $\left(\sls{f_1}{\delta}, \sls{f_1}{-\delta}\right)$ are homotopically equivalent for $\delta>0$ arbitrarily small.
Then by Theorem \ref{mainprelimthm} and the fact that $0 \in H$ is the only critical point with value $0$ (condition \ref{mainthm2_nondeg}) we can compute the critical groups for every $k \in \N$ as
\[C_k(f_i, 0) \cong H_k(\sls{f_i}{\varepsilon/2}, \sls{f_i}{-\varepsilon/2})\] for $i = -1,1$.

The remainder of the proof is standard.
Condition \eqref{mainthm2_specflow} permits to reduce the computation of the critical groups to finite dimensional spaces, and the condition
\[
\mu(L_{-1}) \neq \mu(L_1)
\]
implies that $C_k(f_{-1}, 0)$ is not isomorphic to $C_k(f_{1}, 0)$ for some $k$, which is a contradiction.
\end{potwr}

\section{Strongly Indefinite Functions}
\label{stronglyind}

In this section we will prove a global bifurcation result for a class of nonlinear functionals verifying analogous conditions of Theorem \ref{mainthm2}, except for the fact that the Hessian operators of the functionals have here infinite dimensional negative eigenspaces and thus the Morse index cannot be defined. Condition \ref{mainthm2_specflow} of Theorem \ref{mainthm2} will be replaced in Theorem \ref{mainthm3} by a more general condition involving the spectral flow of the Hessian operators. On the other hand, Thereom \ref{mainthm3} cannot be strictly considered as an extension of Theorem \ref{mainthm2} because it requires a special compactness assumption, as we will see below.



\begin{theorem} 
\label{mainthm3}
Let $f : \R \times H \to \R$ be a $C^2$ function such that
\[
f(\lambda,0) = 0\quad \textrm{and} \quad \nabla f_\lambda( 0) = 0, \quad \quad \forall \lambda \in \R.
\] 
Suppose that, for every $(\lambda,x) \in \R \times H$, one has $\nabla f(\lambda,x) = \J(x) - K(\lambda, x)$
where $\J$ a symmetry of $H$ and the range of $K:\R \times H \to H$ is contained in a compact set. Denote by $I$ the interval $ [-1,1]$ and assume that the following assumptions hold: 
\begin{enumerate}
\item \label{mainthm3_ps} $f$ satisfies the \ps{\R}{\R} condition.
\item \label{mainthm3_partf_bounded} $\partf$ is bounded in the sets of the form $f_{[-N, N]}^{-1}([-N, N])$ for every $N \in \N$.
\item \label{mainthm3_nondeg} For $i = -1,1$, $0 \in H$ is a non-degenerate critical point and the only critical point of $f_i$ with value $0$. 
\item \label{mainthm3_specflow} Assume that, for every $\lambda\in I$, the Hessian of $f$ at zero, 
\[
L_\lambda:= D_X \nabla f(\l , 0)
\]
is Fredholm for every $\l\in [-1,1]$ and suppose
\[
\spf(L,I)\neq 0.
\]
\end{enumerate}
Then the conclusion of Theorem \ref{mainthm} holds.
\end{theorem}

Let $H_n$ be as above and define $f_n$ as the restriction of $f$ to $H_n$. Then clearly we have
\begin{align*}
\nabla f(t,x) &= \J(x) - K(t,x)\\
\nabla f_n(t,x) &= \J(x) - K_n(t,x)\\
d \nabla f_n(t,x) &= \J(x) - K_n'(t,x)
\end{align*}
where $K_n(t,x) = P_n K(t,x)$.

We need three technical lemmas:
\begin{lemma}
\label{approx_eps}
There are $\varepsilon>0$ and $m_0 \in \N$ such that for $m \geq m_0$, the balls $B((\pm 1,0),\varepsilon) \subseteq \R^2$ contain only trivial solutions of $f_m$.
\end{lemma}

\begin{proof}
Assume otherwise, then we have a sequence $(t_n,x_n) \in \R \times H$ such that 
\begin{align*}
x_n \in H_n\setminus H_{n-1}\\
t_n \to t^*=\pm1\\
f(t_n, x_n) \to 0\\
\nabla f(t_n, x_n) = 0.
\end{align*}

Since $\J(H_n) \subseteq H_n$ and since $P_n \to Id$ uniformly in compact sets, we have a sub-sequence we still call $(x_{n})$ such that 
\[x_n = \J P_n K(t_n, x_n) \to x^*.\]

By the continuity of $f$ and $\nabla f$,
\[f(t^*, x^*) = 0, \nabla f(t^*, x^*) = 0.\]
Since property \eqref{mainthm_nonbif} is valid we have $x^* = 0$. Now since $f$ is $C^1$ we have
\[0 = \nabla f_n(t_n, x_n) = d \nabla f_n(t^*, 0).x_n + o(\|x_n\|)\]
\[0 = d \nabla f(t^*, 0).x_n + (Id - P_n).d K(t^*, 0).x_n + o(\|x_n\|).\]
Since the Frechet derivative of a compact function is a compact operator \cite[Theorem 17.1]{KrZa}, and since $P_n \to Id$ uniformly in compact sets, we have (modulo a sub-sequence)
\[\|(Id - P_n).d K(t^*, 0).x_n\| = o(\|x_n\|) \]
so we conclude $d \nabla f(t^*,0).x_n = o(\|x_n\|)$ which contradicts the invertibility of $d \nabla f(t^*,0)$.

The lemma is thus proved.
\end{proof}

The following lemma verifies easily
\begin{lemma}
\label{approx_ps}
For any $m$, the function $f_m$ satisfies $(PS)_{\R,\R}$.
\end{lemma}
\begin{proof}
Assume there is a sequence $(t_n, x_n) \in \R \times H_m$ such that
\[f_m(t_n, x_n) \to c\]
\[\nabla f_m(t_n, x_n) = \J(x_n) - K_m(t_n, x_n) \to 0.\]
Then by the compactness of $K_m$, there is a subsequence (again $x_n$) such that $\J K_m(t_n, x_n)$ converges, and thus
\[x_n = \J(\nabla f_m(t_n, x_n)) + \J K_m(t_n, x_n)\]
also converges.
\end{proof}

Also we prove
\begin{lemma}
\label{approx_limit}
If $(z_n) \subset \R^2$ is a convergent sequence such that $z_n \in S_{f_n}$ then $z = \lim z_n \in S_f$.
\end{lemma}
\begin{proof}
Take $(t_n, x_n) \in \R \times H_n$ such that $(t_n, f(t_n, x_n)) = z_n$ and $\nabla f_n(t_n, x_n) = 0$.
We have
\[x_n = \J K_n(t_n, x_n)\]
so taking a subsequence of $\J K_n(t_n, x_n)$ we have 
\[x = \lim x_n = \J K(t, x)\]
and $(t, f(t, x)) = z$.
\end{proof}

For sufficiently large $m$, the spectral flow is computed as $\spf(f_m, I) = \mu(f_{m,-1}) - \mu(f_{m,1}) \neq 0$ so considering Lemma \ref{approx_ps}, the function $f_m$ is in the hypotheses of Theorem \ref{mainthm2}.
We obtain for every $m \geq m_0$ a family of closed connected sets $C_m \subseteq S_{f_m}$ satisfying the conclusion of Theorem \ref{mainthm}.

\medskip

Now we can prove the main result of this section.

\begin{potwr}{mainthm3}
As in the proof of Theorem \ref{mainthm}, assume none of the alternatives \eqref{mainthm_unbounded}, \eqref{mainthm_intersects} are satisfied.
Take $G_f \subseteq \R^2, R>0, D_+,D_-$ and $\Sout$ as in the proof o Theorem \ref{mainthm}, and $\varepsilon > 0$ as in the Lemma \ref{approx_eps}.
The set $G_f$ is contained in $D_+$, contains $Z_{in}$ and does not intersect $\Sout$.
Again by Whyburn's Lemma there exists an open set $A$ with $\overline A \in D_+$ containing $G_f$ such that 
$\partial A \cap (E_f \cup \Sout) = \emptyset$.

For every $m > m_0$ we have non empty intersections $z_m \in C_m \cap \partial A$.
By Lemma \ref{approx_eps}, we have $z_m \not\in B_{\pm 1}$.

Taking a convergent subsequence given by the compacity of $\partial A$, we may assume $z_m \to z \in \partial A$.
Finally, by Lemma \ref{approx_limit} we obtain $z \in S_f$.
Since $z \not\in \Zin \cup \Sout$ we deduce
$z \in E_f$ which is a contradiction.
\end{potwr}

\end{document}